\documentclass[11pt]{article}

\usepackage[utf8]{inputenc}
\usepackage[T1]{fontenc}
\usepackage[english]{babel}

\usepackage{amsmath, amssymb, amsthm}
\usepackage{ytableau}        
\usepackage{float}
\usepackage{caption}
\usepackage{hyperref}
\usepackage{cleveref}
\usepackage{mathdots}

\usepackage{mathtools}

\usepackage{xcolor}
\usepackage{soul}            
\usepackage{tikz}
\usetikzlibrary{
  arrows.meta,
  arrows,
  calc,
  decorations.pathmorphing,
  decorations.pathreplacing,
  positioning,
  shapes
}
\usepackage{subcaption}

\usepackage{enumitem}

\newcommand\odd{\operatorname{odd}}

\usepackage[backend=biber, style=numeric,maxnames=99]{biblatex}
\usepackage{csquotes}

\usepackage{todonotes}
\usepackage{comment}

\usepackage{geometry}
\theoremstyle{definition}
\newtheorem{theorem}{Theorem}[section]
\newtheorem*{theorem-non}{Theorem}
\newtheorem{lemma}[theorem]{Lemma}

\newtheorem{corollary}[theorem]{Corollary}
\newtheorem{definition}[theorem]{Definition}
\theoremstyle{remark}

\newtheorem{example}[theorem]{Example}
\newtheorem{remark}[theorem]{Remark}

\theoremstyle{definition}
\newtheorem*{theorem*}{Theorem}
\numberwithin{equation}{section}

\newcommand{\coloneq}{\mathrel{\mathop:}=}
\definecolor{turkoise}{RGB}{64,224,208} 

\newcommand\qand{\quad\text{and}\quad}
\definecolor{darkblue}{rgb}{0.0,0,0.7}
\renewcommand\emph[1]{\textcolor{darkblue}{\it #1}}

\title{$A_1^{(1)}$-Grounded partitions at levels $1$ and $2$,\\
Part I: bijections}
\author{Benedek Dombos\thanks{\href{mailto:benedek.dombos@unige.ch}{benedek.dombos@unige.ch}, University of Geneva} \and Jihyeug Jang\thanks{\href{mailto:jihyeug.jang@unige.ch}{jihyeug.jang@unige.ch}, University of Geneva}}
\date{\today}

\begin{document}

\maketitle

\let\oldclearpage\clearpage
\renewcommand{\clearpage}{}

\begin{abstract}
Grounded partitions, introduced by Dousse and Konan, are coloured partitions satisfying difference conditions encoded by a matrix. For suitable choices of this matrix, their generating functions are known to coincide with characters of affine Lie algebras. In this paper, we study, from a combinatorial point of view, the grounded partitions introduced by Dousse, Hardiman and Konan and related to the Lie algebra $A_1^{(1)}$. Using the connection with characters, they showed that the generating function for these grounded partitions is an infinite product. We give direct combinatorial proofs of the corresponding product formulas. In particular, we construct two explicit bijections from grounded partitions to odd overpartitions, and to partitions in which the even parts are distinct.
\end{abstract}

\section{Introduction}

A \emph{partition} of a nonnegative integer \( n \) is a non-increasing sequence \(\lambda = (\lambda_1, \lambda_2, \dots, \lambda_k)\) of positive integers such that $n = \lambda_1 + \lambda_2 + \dots + \lambda_k$. For example, there are five partitions of $4$: \((4), (3,1), (2,2), (2,1,1) \), and \( (1,1,1,1)\). The integers \( \lambda_i \) are called the \emph{parts} of the partition $\lambda$. The number of parts \( k \) is called the \emph{length} of \(\lambda\), and is denoted by \(\ell(\lambda)\). The sum of the terms is called the \emph{size} of $\lambda$, denoted $|\lambda|$. The empty sequence corresponds to the unique partition of size $0$, called the \emph{empty partition}.

In the theory of integer partitions, it is convenient to abbreviate some finite and infinite products by \emph{$q$-Pochhammer symbols}. For formal variables $a$ and $q$, define
\begin{align*}
    (a;q)_k &\coloneq (1 - a)(1 - aq)(1 - aq^2) \cdots (1 - aq^{k-1}), \quad\text{ for \( k \geq 1 \)},\\
     (a;q)_0 &\coloneq 1, \quad (a;q)_\infty \coloneq \prod_{k=0}^{\infty} (1 - aq^k).
\end{align*}

The Rogers--Ramanujan identities were proved as $q$-series identities by Rogers and Ramanujan \cite{RR19}. Their significance comes from the fact that the same identities reappear in several different contexts, including the representation theory of Kac--Moody algebras \cite{LM78,LW84,LW85}, double affine Hecke algebras \cite{CherednikFeigin}, modular forms \cite{BringmannOnoRhoades}, orthogonal polynomials \cite{Bressoud81,GarrettIsmailStanton99}, and solvable models in statistical mechanics \cite{AndrewsBaxterForrester84,Baxter81}.

\begin{theorem}[Rogers--Ramanujan identities, in terms of $q$-series, \cite{RR19}]\label{thm:RR}
\begin{align}\label{eq:RR}
    \sum_{n=0}^{\infty} \frac{q^{n^2}}{(q;q)_n} = \frac{1}{(q;q^5)_\infty (q^4;q^5)_\infty}, \qand
    \sum_{n=0}^{\infty} \frac{q^{n^2+n}}{(q;q)_n} = \frac{1}{(q^2;q^5)_\infty (q^3;q^5)_\infty}.
\end{align}
\end{theorem}

Each of these \(q\)-series identities corresponds to a partition
identity, since both sides are generating functions for certain
classes of partitions. See \cite{Mac16} and \cite{Sch17} for further
details.

\begin{theorem}[Rogers--Ramanujan identities, in terms of partitions]\label{thm:RR2}
  For a positive integer \( n \), the number of partitions of \( n \)
  such that the difference between consecutive parts is at least 2 is
  equal to the number of partitions of \( n \) into parts congruent to
  \( 1 \) or \( 4 \pmod{5} \). Similarly, for a positive integer
  \( n \), the number of partitions of \( n \) into parts greater than
  \( 1 \), with the difference between consecutive parts at least
  \( 2 \), is equal to the number of partitions of \( n \) into parts
  congruent to \( 2 \) or \( 3 \pmod{5} \).
\end{theorem}

Finding an explicit and simple size-preserving bijection for either of
these two partition identities remains notoriously elusive open
problem in enumerative combinatorics.

Lepowsky and Wilson~\cite{LW84,LW85} gave a representation-theoretic
proof of the Rogers--Ramanujan identities, building on earlier work of
Lepowsky and Milne~\cite{LM78}. The product side can be obtained
directly from Lepowsky's product formula (see Lepowsky~\cite{Lep78} or Kac~\cite[Theorem 10.4]{Kac90}) for characters of highest
weight modules of type $A_1^{(1)}$ at level $3$. The representation-theoretic interpretation of the sum side is more
intricate. In these papers, Lepowsky and Wilson developed and applied
the theory of vertex operators to derive it. An alternative approach is to study partitions with difference conditions via the theory of affine and perfect crystals; see Dousse--Konan~\cite{DK19b,DK22}. A representation-theoretic
introduction to this topic is given in the book by Hong and
Kang~\cite{HK02}.

Grounded partitions were introduced by Dousse and Konan
in~\cite{DK19b,DK22}, motivated by the theory of perfect crystals. In
\cite{KMN92}, Kang, Kashiwara, Misra, Miwa, Nakashima, and Nakayashiki
obtained the (KMN)\( ^2 \)-character formula in terms of the energy
function associated with a perfect crystal. Dousse and Konan used this
formula to show that the character of an irreducible highest-weight
module over an affine Lie algebra coincides with the generating
function for the corresponding grounded partitions.

Recently, Dousse, Hardiman, and Konan~\cite{DHK25} defined grounded
partitions of type \( A_1^{(1)} \) at all levels. Throughout this
paper, we consider only grounded partitions of this type and therefore
refer to them simply as grounded partitions, without further mention
of \( A_1^{(1)} \).

\begin{definition}\label{def:grounded_partitions}
  Let \( \ell \) be a nonnegative integer, and let
  \(C_\ell=\{c_0,c_1,\ldots,c_\ell\}\) be a set of \emph{colours}.
  Define \(M_\ell\) to be the square matrix of size \(\ell+1\), with
  rows and columns indexed by \( 0,1, \dots, \ell \), whose
  \((i,j)\)-entry is \(|\ell-i-j|\).
  Let \( c \) be a colour in \( C_\ell \). A \emph{grounded partition}
  at level \(\ell\) with \emph{ground} \(0_{c}\) is a sequence
  \[
    \pi = (\pi_0, \pi_1, \dots, \pi_s),
  \]
  of integers indexed by a colour in \( C_\ell \) with the initial
  part \( \pi_0 = 0_c \), satisfying the exact difference condition:
  for \( j \geq 0 \), if two consecutive parts \(\pi_{j} \) and
  \( \pi_{j+1}\) have sizes \( |\pi_j| \) and \( |\pi_{j+1}| \) with
  colours \(c_{i_j}\) and \(c_{i_{j+1}}\), respectively, then
  \[
    |\pi_{j+1}|-|\pi_{j}| = \text{the \( (i_{j+1},i_j) \)-entry of } M_\ell.
  \]

  We denote by \(\mathcal{P}_{\ell,i}\) the set of all grounded
  partitions at level \( \ell \) with ground \(0_{c_i}\). For a
  grounded partition \( \pi \), the number of nonzero parts is called
  the \emph{length} of \( \pi \), denoted by \( \ell(\pi) \). The
  \emph{size} of \(\pi\) is defined as the size of the underlying
  partition obtained by forgetting the colours, denoted \(|\pi|\). For
  convenience, we omit the ground \(\pi_0\), and write
  \( \pi = (\pi_0, \pi_1, \pi_2, \dots, \pi_s) \) simply as
  \(\pi_1 \pi_2 \cdots\pi_s\).
\end{definition}

The generating function for \( \mathcal{P}_{\ell,i} \) is the
principally specialised character of the irreducible highest-weight
module \( L(\lambda) \) of type \( A_1^{(1)} \) with highest weight
\(\lambda=i\Lambda_0+(\ell-i)\Lambda_1 \), and therefore has a nice
infinite product expression given by Lepowsky’s formula.

\begin{remark}\label{rem:2}
  Partitions are typically written in weakly decreasing order. For
  grounded partitions, however, the sizes of the parts are weakly
  increasing, since the colour of the initial \( 0 \)-part is
  significant. Equivalently, one could define grounded partitions in
  weakly decreasing order by reversing the difference condition and
  setting \( \pi_s = 0_c \) in \Cref{def:grounded_partitions}.
\end{remark}

\begin{remark}\label{rem:3}
  For type \( A_1^{(1)} \), the level of a highest weight
  \( a \Lambda_0 + b \Lambda_1 \) is defined to be \( a+b \). Since
  the highest weight \(i\Lambda_0+(\ell-i)\Lambda_1 \) has level
  \( \ell \), we refer to \( \mathcal{P}_{\ell,i} \) as the set of
  grounded partitions at level \( \ell \).
\end{remark}

\begin{theorem}[{\cite[Theorem~1.6]{DHK25}}]\label{thm:gf_grounded}
  Let \( \ell \) and \( i \) be integers with \( 0 \leq i \leq \ell \).
  Then, the generating function for grounded partitions in
  \( \mathcal{P}_{\ell,i} \) is given by
\[
  \sum_{\pi \in \mathcal{P}_{\ell, i}} q^{|\pi|} = \frac{(q^{i+1};q^{\ell+2})_\infty(q^{\ell-i+1};q^{\ell+2})_\infty(q^{\ell+2};q^{\ell+2})_\infty}{(q;q^2)_\infty (q;q)_\infty}.
\]
\end{theorem}

The equality of the generating functions for
\( \mathcal{P}_{\ell,i} \) and \( \mathcal{P}_{\ell,\ell-i} \) is
immediately obtained from Theorem~\ref{thm:gf_grounded}, but it also
has a straightforward combinatorial proof. Since the matrix $M_\ell$
is symmetric, swapping colours $c_j$ and $c_{\ell-j}$ for all \( j \)
yields a size-preserving bijection between $\mathcal{P}_{\ell,i}$ and
$\mathcal{P}_{\ell,\ell-i}$. Thus, the generating function for
\( \mathcal{P}_{\ell,i} \) is equal to that for
\( \mathcal{P}_{\ell,\ell-i} \). Thus, when considering generating
functions, it suffices to consider \( \mathcal{P}_{\ell,i} \) for
\( 0 \leq i \leq \lfloor \ell/2 \rfloor \). Moreover, the generating
function for \( \mathcal{P}_{\ell,i} \) is equal to the principal
specializations of the characters associated with both highest weights
\( i \Lambda_0 + (\ell-i) \Lambda_1 \) and
\( (\ell-i) \Lambda_0 + i \Lambda_1 \).

Some combinatorial aspects of grounded partitions were studied in~\cite[Section~4]{DHK25}. However, no bijective proof is known for the fact that the generating function of grounded partitions with exact difference conditions is an infinite product after principal specialisation. We believe this direction is worth exploring—not only for the sake of uncovering interesting bijections, but also because grounded partitions with exact difference conditions admit an affine crystal structure, as we show in the second part of this work~\cite{PartII}.

\begin{remark}
  After the first version of our preprint appeared, Kanade and
  Russell~\cite{KR25} found an explicit bijection between tight
  cylindric partitions with \(2\) rows and grounded partitions of type
  \(A_1^{(1)}\) and arbitrary level. Borodin proved a product formula
  for the generating function of cylindric partitions of fixed
  profile~\cite{Bor07}. Thus the bijection of Kanade and Russell gives
  another route to the product formulas for the generating functions
  of grounded partitions.
\end{remark}

For clarity of notation, throughout this paper we relabel the colours
by \( (c_0, c_1, c_2, c_3) = (a,b,c,d) \) and accordingly write
\( (\mathcal{P}_{\ell,a}, \mathcal{P}_{\ell,b}) \) for
\( (\mathcal{P}_{\ell,0}, \mathcal{P}_{\ell,1}) \). We consider the
following three matrices, defining grounded partitions at levels
\( 1 \), \( 2 \), and \( 3 \):
\begin{align*}
M_1=\begin{array}{c|cc}
    & \textcolor{blue}{a} & \textcolor{red}{b} \\ \hline
\textcolor{blue}{a} & 1 & 0 \\
\textcolor{red}{b} & 0 & 1 \\
\end{array},
&&
M_2=\begin{array}{c|ccc}
    & \textcolor{blue}{a} & \textcolor{red}{b} & \textcolor{green!60!black}{c} \\ \hline
\textcolor{blue}{a} & 2 & 1 & 0 \\
\textcolor{red}{b} & 1 & 0 & 1 \\
\textcolor{green!60!black}{c} & 0 & 1 & 2 \\
\end{array},
&&
M_3=\begin{array}{c|cccc}
    & \textcolor{blue}{a} & \textcolor{red}{b} & \textcolor{green!60!black}{c} & \textcolor{yellow!60!orange}{d} \\ \hline
\textcolor{blue}{a} & 3 & 2 & 1 & 0 \\
\textcolor{red}{b} & 2 & 1 & 0 & 1 \\
\textcolor{green!60!black}{c} & 1 & 0 & 1 & 2 \\
\textcolor{yellow!60!orange}{d} & 0 & 1 & 2 & 3 \\
\end{array}.
\end{align*}

We usually omit the ground in the notation at the beginning of each
grounded partition. The ground is always clear from context, as we
specify the set of grounded partitions in advance. For example, we
write \( 0_b \in \mathcal{P}_{2,b} \) for the empty partition, using
the subscript \( b \) to indicate the ground. This is the unique
grounded partition of size zero.

In the following three examples, we consider grounded partitions at
levels \( 1 \), \( 2 \), and \( 3 \), respectively, and use
\Cref{thm:gf_grounded} to obtain their generating functions.

\vspace{1em}
\begin{example}\label{ex:grounded_partitions_1}
\begin{enumerate}
\item We list the elements of $\mathcal{P}_{1,b}$ of size $7$. Here the ground is $0_b$.
  \begin{align*}
    \textcolor{red}{1_b} \, \textcolor{blue}{1_a} \, \textcolor{red}{1_b} \, \textcolor{blue}{1_a} \, \textcolor{red}{1_b} \, \textcolor{blue}{1_a} \, \textcolor{red}{1_b},\quad
    \textcolor{red}{1_b} \, \textcolor{blue}{1_a} \, \textcolor{red}{1_b} \, \textcolor{blue}{1_a} \, \textcolor{red}{1_b} \, \textcolor{red}{2_b},\quad
    \textcolor{red}{1_b} \, \textcolor{blue}{1_a} \, \textcolor{red}{1_b} \, \textcolor{red}{2_b} \, \textcolor{blue}{2_a},\quad \textcolor{red}{1_b} \, \textcolor{red}{2_b} \, \textcolor{blue}{2_a} \, \textcolor{red}{2_b},\quad
    \textcolor{red}{1_b} \, \textcolor{blue}{1_a} \, \textcolor{blue}{2_a} \, \textcolor{blue}{3_a}.
  \end{align*}
  In this case, the colours are completely determined by the size of
  the parts, and by forgetting the colours, we recover the partitions where the difference between consecutive parts is either $0$ or $1$; and $1$ appears as a part for nonempty partitions.
  
  The \emph{conjugate} of a partition $\lambda=(\lambda_1,\ldots,\lambda_r)$ is the partition $\lambda'=(\lambda_1',\ldots,\lambda_{\lambda_1}')$, where $\lambda_j'
    \coloneq
    |\{i\in\{1,\ldots,r\} : \lambda_i\geq j\}|$ for $1\leq j\leq \lambda_1$. The elements of $\mathcal{P}_{1,b}$ are precisely the conjugates of partitions into distinct parts, so their generating function is given by the infinite product
\begin{align}
\label{eq:P1b}\sum_{\pi \in \mathcal{P}_{1,b}} q^{|\pi|} = (-q;q)_\infty = 1 + q + q^2 + 2q^3 + 2q^4 + 3q^5 + 4q^6 + 5q^7 + 6q^8 + \cdots.
\end{align}
\item There is a simple size-preserving bijection between $\mathcal{P}_{1,b}$ and $\mathcal{P}_{1,a}$, by swapping the colours $a$ and $b$. Thus, the generating function of $\mathcal{P}_{1,a}$ is the same infinite product. 
\end{enumerate}
\end{example}

\begin{example}\label{ex:grounded_partitions_2}
\begin{enumerate}
\item We list the elements of $\mathcal{P}_{2,b}$ of size \( 5 \), as they will be used explicitly in Section~\ref{section:bij1}:
\begin{align*}
\textcolor{blue}{1_a} \textcolor{green!60!black}{1_c} \textcolor{blue}{1_a} \textcolor{green!60!black}{1_c} \textcolor{blue}{1_a},\ \textcolor{green!60!black}{1_c} \textcolor{blue}{1_a} \textcolor{green!60!black}{1_c} \textcolor{blue}{1_a} \textcolor{green!60!black}{1_c},\ \textcolor{blue}{1_a} \textcolor{green!60!black}{1_c} \textcolor{blue}{1_a} \textcolor{red}{2_b},\ \textcolor{green!60!black}{1_c} \textcolor{blue}{1_a} \textcolor{green!60!black}{1_c} \textcolor{red}{2_b},\ \textcolor{blue}{1_a} \textcolor{red}{2_b} \textcolor{red}{2_b},\ \textcolor{green!60!black}{1_c} \textcolor{red}{2_b} \textcolor{red}{2_b},\ \textcolor{blue}{1_a} \textcolor{green!60!black}{1_c} \textcolor{green!60!black}{3_c},\ \textcolor{green!60!black}{1_c} \textcolor{blue}{1_a} \textcolor{blue}{3_a}.
\end{align*}
The generating function for \( \mathcal{P}_{2,b} \) is given by the
following infinite product form:
\begin{align}\label{eq:P2b}
  \sum_{\pi \in \mathcal{P}_{2,b}} q^{|\pi|}
  = \frac{(-q;q^2)_\infty}{(q;q^2)_\infty} = 1 + 2q + 2q^2 + 4q^3 + 6q^4 + 8q^5 + 12q^6 + \cdots.
\end{align}
\item We list the elements of $\mathcal{P}_{2,a}$ of size \( 5 \), as
  they will be used explicitly in Section~\ref{section:bij2}:
 \begin{align*}
\textcolor{red}{1_b} \textcolor{blue}{2_a} \textcolor{green!60!black}{2_c},\ \textcolor{blue}{2_a} \textcolor{red}{3_b},\ \textcolor{red}{1_b} \textcolor{green!60!black}{2_c} \textcolor{blue}{2_a},\ \textcolor{red}{1_b} \textcolor{red}{1_b} \textcolor{red}{1_b} \textcolor{blue}{2_a},\ \textcolor{red}{1_b} \textcolor{red}{1_b} \textcolor{red}{1_b} \textcolor{green!60!black}{2_c},\ \textcolor{red}{1_b} \textcolor{red}{1_b} \textcolor{red}{1_b} \textcolor{red}{1_b} \textcolor{red}{1_b}.
\end{align*}
The generating function for \( \mathcal{P}_{2,a} \) is given by
\begin{align}\label{eq:P2a}
\sum_{\pi \in \mathcal{P}_{2,a}} q^{|\pi|} = \frac{(-q^2;q^2)_\infty}{(q;q^2)_\infty} = 1 + q + 2q^2 + 3q^3 + 4q^4 + 6q^5 + 9q^6 + \cdots.
\end{align}
\end{enumerate}
\end{example}

\begin{example}\label{ex:grounded_partitions_3}
  The generating functions for \( \mathcal{P}_{3,b} \) and
  \( \mathcal{P}_{3,a} \) are given by
  \begin{align}
    \label{eq:P31_infiniteprod}
    \sum_{\pi \in \mathcal{P}_{3,b}} q^{|\pi|} &= \frac{1}{(q;q^2)_\infty}\frac{1}{(q;q^5)_\infty(q^4;q^5)_\infty} = 1 + 2q + 3q^2 + 5q^3 + 8q^4 + 12q^5 + 18q^6 + \cdots,\\
    \label{eq:P33_infiniteprod}
\sum_{\pi \in \mathcal{P}_{3,a}} q^{|\pi|} &= \frac{1}{(q;q^2)_\infty}\frac{1}{(q^2;q^5)_\infty(q^3;q^5)_\infty} = 1 + q + 2q^2 + 4q^3 + 5q^4 + 8q^5 + 12q^6 + \cdots.
\end{align}
These are the product sides of the Rogers--Ramanujan identities in
Theorem~\ref{thm:RR}, together with a factor \( 1/(q;q^2)_\infty \). 
\end{example}

At level $1$, the problem reduces to the classical generating function
for partitions into distinct parts. At level $2$, however, obtaining the
generating functions of grounded partitions bijectively is already far
from trivial; this is the main content of the paper. At level $3$, the
generating functions are the product sides of the Rogers--Ramanujan
identities multiplied by an extra factor, and finding bijective proofs
appears to be substantially harder.

In this paper, we provide bijective proofs for the
identities~\eqref{eq:P2b} and \eqref{eq:P2a}, which represent the
generating functions for grounded partitions at level \( 2 \). The
generating functions for \( \mathcal{P}_{2,b} \) and
\( \mathcal{P}_{2,a} \) are the principal specialisations of the
characters of the irreducible highest-weight modules of type
\(A_1^{(1)}\) with highest weights \( \Lambda_0 + \Lambda_1 \) and
\(2\Lambda_0\) (or equivalently \(2\Lambda_1\)), respectively.

In particular, we introduce two bijections for \( \mathcal{P}_{2,b} \)
and \( \mathcal{P}_{2,a} \), respectively. An \emph{overpartition},
introduced by Corteel and Lovejoy~\cite{CL2004}, is a partition where
the first occurrence of each number may be overlined. In other words,
the overlined parts form a partition into distinct parts and the
non-overlined parts form a classical partition. The first bijection
goes between grounded partitions in \( \mathcal{P}_{2,b} \) and
overpartitions into odd parts.

\begin{theorem}\label{thm:main_bij1}
  Let \(\mathcal{P}_b(n,k)\) denote the set of grounded partitions in
  \(\mathcal{P}_{2,b}\) of size \(n\) with \(k\) odd parts, and let
  \( \overline{\mathcal{PO}}(n,k) \) denote the set of overpartitions into odd
  parts of size \( n \) and length \( k \).
  Then, for all integers \( n \geq k \geq 0 \), there is a bijection between $\mathcal{P}_b(n,k)$ and $\overline{\mathcal{PO}}(n,k)$. In particular, we have
  \[
    |\mathcal{P}_b(n,k)| = |\overline{\mathcal{PO}}(n,k)|.
  \]
\end{theorem}

The second bijection for \( \mathcal{P}_{2,a} \) is obtained by
modifying the first bijection.

\begin{theorem}\label{thm:main_bij2}
  Let \(\mathcal{P}_a(n,k)\) denote the set of grounded partitions in
  \(\mathcal{P}_{2,a}\) of size \(n\) with \(k\) odd parts, and let
  \( \mathcal{E}(n,k) \) denote the set of partitions of size
  \( n \) with \( k \) odd parts such that all even parts are
  distinct. Then, for all integers \( n \geq k \geq 0 \), there is a bijection between $\mathcal{P}_a(n,k)$ and $\mathcal{E}(n,k)$. In particular, we have
  \[
    |\mathcal{P}_a(n,k)| = |\mathcal{E}(n,k)|.
  \]
\end{theorem}

Using these bijections, we obtain the following \( q \)-series
identities involving an additional parameter \( t \). Their produce
sides can be viewed as refinements of the produce sides in
\eqref{eq:P2b} and \eqref{eq:P2a}, respectively.
\begin{align}
\label{eq:P2b_id}\sum_{n \geq 0} \frac{t^n q^{n^2} (-1; q^2)_{n}}{(t q; q^2)_n (q^2; q^2)_n}
    &= \frac{(-t q;q^2)_\infty}{(t q; q^2)_\infty},\\
\label{eq:P2a_id}\sum_{n\geq0}\frac{q^{n(n+1)}\,(-t q;q^2)_n}{(q^2;q^2)_{n}\,(t q;q^2)_{n+1}}
  &=\frac{(-q^2;q^2)_\infty}{(t q;q^2)_\infty}.
\end{align}

The remainder of this paper is organized as follows. In
\Cref{section:bij1}, we prove \Cref{thm:main_bij1} and
give a bijective proof of \eqref{eq:P2b}. We then use properties of
grounded partitions in \( \mathcal{P}_{2,b} \) to give a combinatorial
proof of \eqref{eq:P2b_id}. \Cref{section:bij2} follows a similar
structure: we prove \Cref{thm:main_bij2}, give a bijective
proof of \eqref{eq:P2a}, and then prove \eqref{eq:P2a_id}
combinatorially.

\section*{Acknowledgements}

The authors warmly thank Jehanne Dousse for her guidance throughout as well as her valuable comments on the earlier versions of this paper. We are also grateful to Thomas Gerber, Frédéric Jouhet, Christian Krattenthaler, Philippe Nadeau and Ali Uncu for valuable discussions and insightful comments that contributed to the development of this work. The authors are funded by the SNSF Eccellenza grant of Jehanne Dousse, PCEFP2 202784.


\section{Bijection between $\mathcal{P}_{2,b}$ and odd overpartitions}\label{section:bij1}

In this section, we prove \Cref{thm:main_bij1} and
obtain \eqref{eq:P2b} as an immediate consequence. We first introduce
a combinatorial object corresponding to the right-hand side of
\eqref{eq:P2b}, namely overpartitions into odd parts, and then
establish a bijection with grounded partitions in
\( \mathcal{P}_{2,b} \).

Let \( \overline{\mathcal{P}} \) denote the
set of overpartitions. It is well-known~\cite{CL2004}
that the generating function of overpartitions is given by
\[
  \sum_{\lambda \in \overline{\mathcal{P}}} q^{|\lambda|} = \frac{(-q;q)_\infty}{(q;q)_\infty}.
\]
We consider overpartitions in which all parts are odd.
Let \( \overline{\mathcal{PO}} \) denote the set of overpartitions into odd
parts. This case is closely related to known results, and the corresponding generating function that keeps track of the length follows readily:
\begin{equation}\label{eq:gf_odd_overpar}
  \sum_{\lambda \in \overline{\mathcal{PO}}} t^{\ell(\lambda)}q^{|\lambda|} = \frac{(-tq;q^2)_\infty}{(tq;q^2)_\infty}.
\end{equation}

\begin{example}
  We list all grounded partitions in $\mathcal{P}_{2,b}$ of sizes
  \( 5 \) in \Cref{ex:grounded_partitions_2}~(i), which we can
  compare with the list of all overpartitions into odd parts of size
  \( 5 \):
\[
(1,1,1,1,1), \ (\overline{1},1,1,1,1), \ (3,1,1), \ (3,\overline{1},1), \ (\overline{3},1,1), \ (\overline{3},\overline{1},1), \ (5), \ (\overline{5}).
\]
\end{example}

We seek a size-preserving bijection between these sets. A more refined version, which also fixes the number of odd parts, is not only easier to construct but also yields a stronger result.

We begin by observing the grounded partitions in
\( \mathcal{P}_{2,b} \). All even parts are coloured \( b \). Consider
an odd part \( 2i+1 \). If \( 2i \) appears in the grounded partition,
then the first \( 2i+1 \) can be coloured either \( a \) or \( c \);
otherwise, the first \( 2i+1 \) has the same colour as that of the
last part \( 2i-1 \). Once the colour of the first occurrence of an
odd part is chosen, the colours of all other parts of the same size
alternate between \( a \) and \( c \). Even parts may or may not
appear, but each odd part must appear at least once. The following
bijection is the main result of this section.

\begin{proof}[Proof of \Cref{thm:main_bij1}]
  We begin by constructing a map from \( \mathcal{P}_b(n,k) \) to
  \( \overline{\mathcal{PO}}(n,k) \). The idea is to add the even
  parts to the odd parts of the grounded partitions in such a way that
  the colour sequence of the odd parts remains unchanged, not only the
  number of odd parts.
  Throughout the proof, all partitions and overpartitions are written in weakly increasing order.
  
  \underline{\it{Step 1:}} Given a grounded partition
  \( \pi \in \mathcal{P}_b(n,k) \), we first define the corresponding
  \emph{minimal grounded partition}
  \( \pi_{\mathrm{min}} \in \mathcal{P}_{2,b} \) with \( k \) odd
  parts as follows. We remove as many even parts as possible from
  \( \pi \), such that the resulting partition still belongs to
  \(\mathcal{P}_{2,b}\) and has \( k \) odd parts. More explicitly,
  suppose that at least one even part \( (2i)_b \) appears between two
  consecutive odd parts, namely
  \( (2i-1)_x(2i)_b \cdots (2i)_b (2i+1)_y \), where
  \( x,y \in \{a,c\} \). If the two odd parts have the same colour,
  that is, \( x=y=a \) or \( x=y=c \), then all such even parts are
  removed, because $(2i-1)_x(2i+1)_y$ still satisfies the difference
  conditions of $M_2$. If they have different colours, namely
  \( (x,y) = (a,c) \) or \( (c,a) \), then all but one even part
  \( (2i)_b \) are removed, because $(2i-1)_x(2i+1)_y$ does not
  satisfy the difference conditions but $(2i-1)_x(2i)_b(2i+1)_y$ does.
  We denote the resulting grounded partition by
  \( \pi_{\mathrm{min}} \), and the partition consisting of the
  removed even parts by \( \pi_e \), where we also removed the colour
  $b$ from the parts. We call the even parts removed from $\pi$ the
  \emph{loose parts}. This step is illustrated in
  Figure~\ref{fig:ground_to_overpart_1}.
  \begin{figure}[H]
    \centering
    \[
      \underbrace{\textcolor{green!60!black}{1_c}\textcolor{blue}{1_a}\textcolor{green!60!black}{1_c}
        \textcolor{red}{2_b}\textcolor{red}{2_b}
        \textcolor{blue}{3_a}\textcolor{green!60!black}{3_c}
        \textcolor{green!60!black}{5_c}\textcolor{blue}{5_a}
        \textcolor{red}{6_b}\textcolor{red}{6_b}
        \textcolor{blue}{7_a}\textcolor{green!60!black}{7_c}\textcolor{blue}{7_a}\textcolor{green!60!black}{7_c}
        \textcolor{red}{8_b}
        \textcolor{blue}{9_a}}_{\pi}
      \quad \mapsto \quad
      \left(
        \underbrace{\textcolor{green!60!black}{1_c}\textcolor{blue}{1_a}\textcolor{green!60!black}{1_c}
          \textcolor{red}{2_b}
          \textcolor{blue}{3_a}\textcolor{green!60!black}{3_c}
          \textcolor{green!60!black}{5_c}\textcolor{blue}{5_a}
          \textcolor{blue}{7_a}\textcolor{green!60!black}{7_c}\textcolor{blue}{7_a}\textcolor{green!60!black}{7_c}
          \textcolor{red}{8_b}
          \textcolor{blue}{9_a}}_{\pi_{\mathrm{min}}}
        \quad,\quad
        \underbrace{(2,6,6)}_{\pi_e}
      \right)%
    \]
    \caption{From grounded partition to overpartition: step 1}
    \label{fig:ground_to_overpart_1}
  \end{figure}

  \underline{\it{Step 2:}} Using \( \pi_{\mathrm{min}} \) and
  \( \pi_e \), we construct an overpartition into odd parts. For each
  even part \( 2i \) in \( \pi_{\mathrm{min}} \), starting from the
  smallest, we increase the \( i \) odd parts to the left of \( 2i \)
  by \( 2 \). By definition of \(\mathcal{P}_{2,b} \), each odd part
  between \(1\) and \(2i-1\) must appear at least once, so there are
  at least \(i\) odd parts to the left of \(2i\). Let \( \pi^{(1)} \)
  be the resulting partition. Note that all parts of \( \pi^{(1)} \)
  are odd, though \( \pi^{(1)} \) may no longer belong to
  \( \mathcal{P}_{2,b} \). This step is illustrated in
  Figure~\ref{fig:ground_to_overpart_2}.
  \begin{figure}[H]
    \centering
    \scalebox{0.9}{%
      \begin{tikzpicture}[baseline=(current bounding box.center), every node/.style={anchor=base}]
        \node (n1)  at (-0.1,0)   {$\Bigg(\, \textcolor{green!60!black}{1_c}$};
        \node (n2)  at (0.5,0) {$\textcolor{blue}{1_a}$};
        \node (n3)  at (1.0,0) {$\textcolor{green!60!black}{1_c}$};
        \node (n4)  at (1.5,0) {$\textcolor{red}{2_b}$};
        \node (n5)  at (2.0,0) {$\textcolor{blue}{3_a}$};
        \node (n6)  at (2.5,0) {$\textcolor{green!60!black}{3_c}$};
        \node (n7)  at (3.0,0) {$\textcolor{green!60!black}{5_c}$};
        \node (n8)  at (3.5,0) {$\textcolor{blue}{5_a}$};
        \node (n9)  at (4.0,0) {$\textcolor{blue}{7_a}$};
        \node (n10) at (4.5,0) {$\textcolor{green!60!black}{7_c}$};
        \node (n11) at (5.0,0) {$\textcolor{blue}{7_a}$};
        \node (n12) at (5.5,0) {$\textcolor{green!60!black}{7_c}$};
        \node (n13) at (6.0,0) {$\textcolor{red}{8_b}$};
        \node (n14) at (6.5,0) {$\textcolor{blue}{9_a}$,\,\,};
        
        \draw[->, red, thick, bend right=40] (n4.north) to (n3.north);
        \draw[->, red, thick, bend right=40] (n13.north) to (n9.north);
        \draw[->, red, thick, bend right=45] (n13.north) to (n10.north);
        \draw[->, red, thick, bend right=50] (n13.north) to (n11.north);
        \draw[->, red, thick, bend right=55] (n13.north) to (n12.north);
        \node at (3.25,-0.1) {\(\underbrace{\hspace{6.8cm}}_{\pi_{\mathrm{min}}}\)};

        \node at (7.8,0) {$\underbrace{(2,6,6)}_{\pi_e} \,\,\Bigg)$};

        \node at (9.0,0) {$\displaystyle \mapsto$};
        \node (r1)  at (13.5,0)  {$      \Bigg(
      \underbrace{\textcolor{green!60!black}{1_c}\,\textcolor{blue}{1_a}\,\textcolor{green!60!black}{3_c}\,\textcolor{blue}{3_a}\,\textcolor{green!60!black}{3_c}\,\textcolor{green!60!black}{5_c}\,\textcolor{blue}{5_a}\,\textcolor{blue}{9_a}\,\textcolor{green!60!black}{9_c}\,\textcolor{blue}{9_a}\,\textcolor{green!60!black}{9_c}\,\textcolor{blue}{9_a}}_{\pi^{(1)}}
      \quad , \quad
      \underbrace{(2,6,6)}_{\pi_e}
      \Bigg)$};
      \end{tikzpicture}
    }
    \caption{From grounded partition to overpartition: step 2}
    \label{fig:ground_to_overpart_2}
  \end{figure}

  \underline{\it{Step 3:}} We now proceed to add the parts of
  \( \pi_e \) to those of \( \pi^{(1)} \). For each even part \( 2i \) in
  \( \pi_e \), we increase the largest \( i \) odd parts of \( \pi^{(1)} \) by
  \( 2 \) (the order of increasing these parts does not matter). Let
  \( \pi^{(2)} \) denote the resulting coloured partition. This step is illustrated in
  Figure~\ref{fig:ground_to_overpart_3}.
  \begin{figure}[H]
    \centering
    \begin{align*}
      \left(
      \underbrace{\textcolor{green!60!black}{1_c}\,\textcolor{blue}{1_a}\,\textcolor{green!60!black}{3_c}\,\textcolor{blue}{3_a}\,\textcolor{green!60!black}{3_c}\,\textcolor{green!60!black}{5_c}\,\textcolor{blue}{5_a}\,\textcolor{blue}{9_a}\,\textcolor{green!60!black}{9_c}\,\textcolor{blue}{9_a}\,\textcolor{green!60!black}{9_c}\,\textcolor{blue}{9_a}}_{\pi^{(1)}}
      \quad,\quad
      \underbrace{(2,6,6)}_{\pi_e}
      \right)
      \quad \mapsto \quad
      \underbrace{\textcolor{green!60!black}{1_c}\,\textcolor{blue}{1_a}\,\textcolor{green!60!black}{3_c}\,\textcolor{blue}{3_a}\,\textcolor{green!60!black}{3_c}\,\textcolor{green!60!black}{5_c}\,\textcolor{blue}{5_a}\,\textcolor{blue}{9_a}\,\textcolor{green!60!black}{9_c}\,\textcolor{blue}{13_a}\,\textcolor{green!60!black}{13_c}\,\textcolor{blue}{15_a}}_{\pi^{(2)}}
    \end{align*}
    \caption{From grounded partition to overpartition: step 3}
    \label{fig:ground_to_overpart_3}
  \end{figure}

  \underline{\it{Step 4:}} Finally, for the first occurrence of each value in \( \pi^{(2)} \), we place an overline when its colour is \( c \), and
  leave it non-overlined when its colour is \( a \). By removing all
  colours from \( \pi^{(2)} \), we obtain an overpartition into odd parts
  of length \( k \) from \( \pi^{(2)} \). This final step is illustrated
  in Figure~\ref{fig:ground_to_overpart_4}.
  \begin{figure}[H]
    \centering
    \begin{align*} 
      \underbrace{\textcolor{green!60!black}{1_c}\,\textcolor{blue}{1_a}\,\textcolor{green!60!black}{3_c}\,\textcolor{blue}{3_a}\,\textcolor{green!60!black}{3_c}\,\textcolor{green!60!black}{5_c}\,\textcolor{blue}{5_a}\,\textcolor{blue}{9_a}\,\textcolor{green!60!black}{9_c}\,\textcolor{blue}{13_a}\,\textcolor{green!60!black}{13_c}\,\textcolor{blue}{15_a}}_{\pi^{(2)}}
      \quad \mapsto \quad \underbrace{(\overline{1},1,\overline{3},3,3,\overline{5},5,9,9,13,13,15)}_{\lambda}
    \end{align*}
    \caption{From grounded partition to overpartition: step 4}
    \label{fig:ground_to_overpart_4}
  \end{figure}

  \underline{\it{Inverse map, Step 1:}} We now define a map from
  \( \overline{\mathcal{PO}}(n,k) \) to \( \mathcal{P}_b(n,k) \),
  which is the inverse of the map defined above. Let
  \( \lambda = (\lambda_1 \leq \cdots \leq \lambda_k) \) be an
  overpartition into odd parts of length \( k \). From
  \( \lambda \), we first construct a grounded partition \( \pi_s \)
  with only odd parts and a specified colour sequence. For a maximal
  sequence \( \lambda_i = \cdots = \lambda_j \) with $i\leq j$ of
  equal parts, we assign colour \( a \) to the part \( \lambda_i \) if
  it is non-overlined, and \( c \) if it is overlined. The remaining parts
  \( \lambda_{i+1},\dots,\lambda_j \) are coloured alternately in
  \( a \) and \( c \). This gives a colour sequence
  \( (s_1,\dots,s_k) \), where \( \lambda_i \) is coloured \( s_i \).
  There is a unique grounded partition \( \pi_s \in \mathcal{P}_{2,b} \)
  with this colour sequence. Indeed, the first part is not coloured
  $b$, and afterwards repeated colours increase the part size by two,
  while alternating colours leave the size unchanged. This step is
  illustrated in Figure~\ref{fig:overpart_to_grounded_1}.
\begin{figure}[H]
\centering
\begin{minipage}{15cm}
\[
\underbrace{
(\overline{1},1,\overline{3},3,3,\overline{5},5,9,9,13,13,15)}_{\lambda} \quad \mapsto \quad \underbrace{ca\,cac\,ca\,ac\,ac\,a}_{(s_1,\dots,s_k)}\quad \mapsto \quad
\underbrace{\textcolor{green!60!black}{1_c}\,
\textcolor{blue}{1_a}\,
\textcolor{green!60!black}{1_c}\,
\textcolor{blue}{1_a}\,
\textcolor{green!60!black}{1_c}\,
\textcolor{green!60!black}{3_c}\,
\textcolor{blue}{3_a}\,
\textcolor{blue}{5_a}\,
\textcolor{green!60!black}{5_c}\,
\textcolor{blue}{5_a}\,
\textcolor{green!60!black}{5_c}\,
\textcolor{blue}{5_a}}_{\pi_s}
\]
\end{minipage}
\caption{From overpartition to grounded partition: step 1}
\label{fig:overpart_to_grounded_1}
\end{figure}

\underline{\it{Step 2:}} We construct an array of \( 2 \)s beneath
the grounded partition \( \pi_s \), arranged so that the sum of each
column gives the corresponding part of the original overpartition
\( \lambda \). This array encodes how much each odd part in \( \pi_s \)
must be increased by to reconstruct \( \lambda \). See
Figure~\ref{fig:overpart_to_grounded_2}.
\begin{figure}[h]
\centering

\begin{minipage}{15cm}
\[
(\overline{1},1,\overline{3},3,3,\overline{5},5,9,9,13,13,15) \quad \mapsto \quad \begin{array}{cccccccccccc}
\textcolor{green!60!black}{1_c} & 
\textcolor{blue}{1_a} & 
\textcolor{green!60!black}{1_c} & 
\textcolor{blue}{1_a} & 
\textcolor{green!60!black}{1_c} & 
\textcolor{green!60!black}{3_c} & 
\textcolor{blue}{3_a} & 
\textcolor{blue}{5_a} & 
\textcolor{green!60!black}{5_c} & 
\textcolor{blue}{5_a} & 
\textcolor{green!60!black}{5_c} & 
\textcolor{blue}{5_a} \\
 & &  \textcolor{red}{2} & \textcolor{red}{2} & \textcolor{red}{2} & \textcolor{red}{2} & 
 \textcolor{red}{2} & 
 \textcolor{red}{2} & 
 \textcolor{red}{2} & 
 \textcolor{red}{2} & 
 \textcolor{red}{2} & 
 \textcolor{red}{2} \\
  & & & & & & & 
 \textcolor{red}{2} & 
 \textcolor{red}{2} & 
 \textcolor{red}{2} & 
 \textcolor{red}{2} & 
 \textcolor{red}{2} \\
 & & & & & & & & & \textcolor{red}{2} & \textcolor{red}{2} & 
 \textcolor{red}{2} \\
 & & & & & & & & & \textcolor{red}{2} & \textcolor{red}{2} & 
 \textcolor{red}{2} \\
 & & & & & & & & & & & 
 \textcolor{red}{2}
\end{array}
\]
\end{minipage}
\caption{From overpartition to grounded partition: step 2}
\label{fig:overpart_to_grounded_2}
\end{figure}

\underline{\it{Step 3:}} This is a recursive construction. At
each iteration, let \( m \) be the number of columns in the array of $2$s,
and let \( 2j-1 \) be the largest part in the current partition. In
Figure~\ref{fig:overpart_to_grounded_2}, we have $m=10$ and $j=3$.
\begin{enumerate}
\item First, assume $m > j$. We insert a non-loose even part into the grounded partition. There
  exists a unique gap between parts \( (2i-1)_x \) and \( (2i-1)_y \)
  with $(x,y)=(a,c)$ or $(c,a)$, such that the number of \( 2 \)s in
  the top row of the array—read from the left up to and including the
  column below \( (2i - 1)_a \)—is exactly \( i \). Indeed, the number of \(2\)s up to and including the largest odd part of size \(2j-1\) is at least \(j\), while the number of \(2\)s up to and including the first part of size \(1\) is at most \(1\). Moving from right to left starting at the part of size \(2j-1\), the count of \(2\)s decreases by \(1\) at each step; the odd part stays the same if and only if the colours \(a\) and \(c\) alternate, otherwise it decreases by \(2\).
  
  We insert
  \( (2i)_b \) into this gap, increase the following \( m-i \) odd
  parts by \( 2 \), and remove the top row of the array of $2$s.
\item If \( m \leq j \), we stop.
\end{enumerate}

The first iteration of this step is illustrated in
Figures~\ref{fig:overpart_to_grounded_3.1}, where we have $m=5$ and
$j=4$.
\begin{figure}[H]
\centering
\begin{minipage}{15cm}
\begin{align*}
&\mapsto \quad \scalebox{1}{$\begin{array}{ccccccccccccc}
\textcolor{green!60!black}{1_c} & 
\textcolor{blue}{1_a} & 
\textcolor{green!60!black}{1_c} & 
\textcolor{red}{2_b} & 
\textcolor{blue}{3_a} & 
\textcolor{green!60!black}{3_c} & 
\textcolor{green!60!black}{5_c} & 
\textcolor{blue}{5_a} & 
\textcolor{blue}{7_a} & 
\textcolor{green!60!black}{7_c} & 
\textcolor{blue}{7_a} & 
\textcolor{green!60!black}{7_c} & 
\textcolor{blue}{7_a} \\
  & & & & & & & & 
 \textcolor{red}{2} & 
 \textcolor{red}{2} & 
 \textcolor{red}{2} & 
 \textcolor{red}{2} & 
 \textcolor{red}{2} \\
 & & & & & & & & & & \textcolor{red}{2} & \textcolor{red}{2} & 
 \textcolor{red}{2} \\
 & & & & & & & & & & \textcolor{red}{2} & \textcolor{red}{2} & 
 \textcolor{red}{2} \\
 & & & & & & & & & & & & 
 \textcolor{red}{2}
\end{array}$}
\end{align*}
\end{minipage}
\caption{From overpartition to grounded partition: step 3, first iteration}
\label{fig:overpart_to_grounded_3.1}
\end{figure}

The second iteration of this step is illustrated in
Figures~\ref{fig:overpart_to_grounded_3.2}, where we have $m=3$ and $j=5$, so we stop.
\begin{figure}[H]
\centering
\begin{minipage}{15cm}
\begin{align*}
&\mapsto \quad \scalebox{1}{$\begin{array}{cccccccccccccc}
\textcolor{green!60!black}{1_c} & 
\textcolor{blue}{1_a} & 
\textcolor{green!60!black}{1_c} & 
\textcolor{red}{2_b} & 
\textcolor{blue}{3_a} & 
\textcolor{green!60!black}{3_c} & 
\textcolor{green!60!black}{5_c} & 
\textcolor{blue}{5_a} & 
\textcolor{blue}{7_a} & 
\textcolor{green!60!black}{7_c} & 
\textcolor{blue}{7_a} & 
\textcolor{green!60!black}{7_c} & 
\textcolor{red}{8_b} & 
\textcolor{blue}{9_a} \\
 & & & & & & & & & & \textcolor{red}{2} & \textcolor{red}{2} & & 
 \textcolor{red}{2} \\
 & & & & & & & & & & \textcolor{red}{2} & \textcolor{red}{2} & &
 \textcolor{red}{2} \\
 & & & & & & & & & & & & & 
 \textcolor{red}{2}
\end{array}$}
\end{align*}
\end{minipage}
\caption{From overpartition to grounded partition: step 3, second iteration}
\label{fig:overpart_to_grounded_3.2}
\end{figure}

Let \( \pi_{\mathrm{min}} \) denote the resulting grounded partition
in the top row, and let \( \pi_e \) denote the partition whose parts
are given by the row sums of the remaining array of \( 2 \)s. From
this step, we obtain \( (\pi_{\mathrm{min}}, \pi_e) \).

\underline{\it{Step 4:}} As the reverse of the procedure described in
Figure~\ref{fig:ground_to_overpart_1}, we obtain a grounded partition
\( \pi \in \mathcal{P}_b(n,k) \) from \( \pi_{\mathrm{min}} \) by
assigning color \( b \) to every part of \( \pi_e \) and inserting
these parts into appropriate positions, that is positions that
preserve the grounded partition.

In the forward map, the procedure for obtaining
\( \lambda \in \overline{\mathcal{PO}}(n,k) \) from
\( \pi \in \mathcal{P}_b(n,k) \) can be written as
\[
  \pi \longmapsto (\pi_{\mathrm{min}}, \pi_e) \longmapsto \lambda.
\]
The first map is clearly reversible. By the construction of the
inverse map, one can easily see that the second map is also
reversible. Since the two maps in the process are reversible, the two
sets \( \mathcal{P}_b(n,k) \) and \( \overline{\mathcal{PO}}(n,k) \)
are in bijection, which completes the proof.
\end{proof}

By Theorem~\ref{thm:main_bij1} and the generating
function~\eqref{eq:gf_odd_overpar} with \( t = 1 \), we arrive at a
bijective proof of \eqref{eq:P2b}.

We now prove \eqref{eq:P2b_id} using properties of grounded partitions
in \( \mathcal{P}_{2,b} \). For a grounded partition \( \pi \), let
\( \odd(\pi) \) denote the number of odd parts in \( \pi \). We first
obtain the following bivariate generating function.

\begin{lemma}\label{lem:gf_P2b_sum}
  We have
  \[
    \sum_{\pi \in \mathcal{P}_{2,b}} t^{\odd(\pi)} q^{|\pi|} = \sum_{n \geq 0} \frac{t^n q^{n^2} (-1; q^2)_{n}}{(t q; q^2)_n (q^2; q^2)_n}.
  \]
\end{lemma}
\begin{proof}
  Let \( 2n-1 \) be the largest odd part of
  \( \pi \in \mathcal{P}_{2,b} \). Then each odd part \(1,3,\dots,2n-1 \)
  appears at least once, contributing the factor
  \( t^n q^{n^2}/(tq;q^2)_n \) without keeping track of the colours.

  In a grounded partition \(\pi \in \mathcal{P}_{2,b} \), all even
  parts are coloured \(b\), and all odd parts are coloured \(a\) or
  \(c\), alternatingly along each number \(2i-1\) for
  \( 1\leq i \leq n \). Thus the colour of the first occurrence of
  \(2i-1\) determines the colour of all parts \(2i-1\). One can encode
  the colour of the first \(2i-1\) by overlining or leaving non-overlined
  the preceding even number \(2i\) (whether or not it appears as a
  part). This contributes the factor \( (-1;q^2)_n\). The even parts
  in \(\pi\) are generated by \(1/(q^2;q^2)_n\), without keeping track
  of the colours. Combining these contributions yields the desired
  generating function.
\end{proof}

\begin{corollary}\label{thm:P2b_id2}
  We have
  \[
    \sum_{n \geq 0} \frac{t^n q^{n^2} (-1; q^2)_{n}}{(t q; q^2)_n (q^2; q^2)_n}
    = \frac{(-t q;q^2)_\infty}{(t q; q^2)_\infty}.
  \]
\end{corollary}
\begin{proof}
  This follows from the generating functions for two partition models
  in Lemma~\ref{lem:gf_P2b_sum} and \eqref{eq:gf_odd_overpar}, together
  with the bijection between the two models in
  Theorem~\ref{thm:main_bij1}.
\end{proof}

\section{Bijection between $\mathcal{P}_{2,a}$ and partitions with distinct even parts}\label{section:bij2}

In this section we prove \Cref{thm:main_bij2} and obtain
\eqref{eq:P2a} combinatorially. The right-hand side of \eqref{eq:P2a}
is the generating function of partitions where even parts do not
repeat. Apart from this distinction, the overall structure of this
section follows that of Section~\ref{section:bij1}.

Let $\mathcal{E}$ denote
the set of partitions where all even parts are distinct. Then, we have
\begin{equation}\label{eq:gf_dist_even}
  \sum_{\lambda \in \mathcal{E}} t^{\odd(\lambda)} q^{|\lambda|} = \frac{(-q^2;q^2)_\infty}{(tq;q^2)_\infty}.
\end{equation}

\begin{example}
We list all grounded partitions in $\mathcal{P}_{2,a}$ of size \( 5 \) in \Cref{ex:grounded_partitions_2}~(ii), which we can compare with the list of all partitions in $\mathcal{E}$ of size \( 5 \):
\[
  (1,1,1,1,1), \ (2,1,1,1), \ (3,1,1), \ (3,2), \ (4,1), \ (5).
\]
\end{example}

\begin{remark}
  There are many different ways to write the product side in terms of $q$-Pochhammer
  symbols. For example, we have:
  \[
    \frac{(-q^2;q^2)_\infty}{(q;q^2)_\infty}
    = \frac{(q^4;q^4)_\infty}{(q;q)_\infty}.
  \]
  The right-hand side is the generating function for partitions with no part divisible by $4$. It seems to us that the product
  on the left-hand side is more convenient for finding a bijection.
\end{remark}

For a grounded partition in \( \mathcal{P}_{2,a} \), all odd parts
have colour \( b \), (unlike in \( \mathcal{P}_{2,b} \)). An even part
\( 2i \) can be coloured either \( a \) or \( c \) if the part
\( 2i-1 \) appears in the partition. If not, then the first occurrence
of \( 2i \) has the same colour as the colour of the last \( 2i-2 \).
Afterwards, the colours of all parts of size \( 2i \) alternate
\( a \) and \( c \). Odd parts may or may not appear, but each even
part must appear at least once.

\begin{proof}[Proof of \Cref{thm:main_bij2}]
  A modification of the bijective proof of Theorem~\ref{thm:main_bij1}
  works. Throughout the proof, all partitions and overpartitions are
  written in weakly increasing order. We construct a map from
  \( \mathcal{P}_a(n,k)\) to \( \mathcal{E}(n,k) \).

  \underline{\it{Step 1:}} The first step is a notational trick. Let
  \(\pi\in\mathcal{P}_a(n,k)\). If an even part \(2i\) in \(\pi\) with
  colour \(c\) follows an odd part \(2i-1\), we put an overline over
  it. If an even part \(2i\) with colour \(a\) follows an odd part
  \(2i-1\), we leave it non-overlined. Furthermore, whenever there are
  two consecutive parts \(2i\) and \(2i+2\) in \(\pi\), we put an
  overline over the first occurrence of \(2i+2\). Finally, we remove
  the colours to obtain an overpartition \(\widetilde{\pi}\). We
  denote the set \(\{\widetilde{\pi}:\pi\in\mathcal{P}_a(n,k)\}\) by
  \(\widetilde{\mathcal{P}}_a(n,k)\), which can be characterised as
  follows. \(\widetilde{\mathcal{P}}_a(n,k)\) is the set of
  overpartitions of size \(n\) with \(k\) odd parts, where the
  difference between consecutive parts is \(0,1\), or \(2\); the
  difference between consecutive odd parts is \(0\); only even parts
  can be overlined; and the first occurrence of an even number is
  overlined if and only if the immediately preceding part is not odd.
  For example, if
  \( \pi = (1_b, 1_b, 1_b, 2_c, 2_a, 2_c, 4_c, 4_a, 6_a, 7_b, 7_b,
  7_b, 8_a, 8_c, 9_b, 9_b, 10_a, 11_b) \) in \(\mathcal{P}_a(99,9) \),
  then
\[
\widetilde{\pi} = (1,1,1,\overline{2},2,2,\overline{4},4,\overline{6},7,7,7,8,8,9,9,10,11) \in \widetilde{\mathcal{P}}_a(99,9).
\]
The map
\( \pi \mapsto \widetilde{\pi} \) is a bijection from
\(\mathcal{P}_a(n,k)\) to \(\widetilde{\mathcal{P}}_a(n,k)\). 

\underline{\it{Step 2:}} The role of the odd parts in the proof of
Theorem~\ref{thm:main_bij1} is replaced by odd parts and overlined even
parts. Take a partition \(\pi \in \mathcal{P}_{a}(n,k)\) and let
\( \widetilde{\pi} \in \widetilde{\mathcal{P}}_a(n,k) \) be the
corresponding overpartition. Define the \emph{minimal overpartition}
\(\pi_{\mathrm{min}} \) to be overpartition obtained from
\( \widetilde{\pi} \) as follows: for each even part size, remove all
non-overlined parts of that size if the corresponding overlined part
occurs, and otherwise remove all but one non-overlined part of that
size. As before, we call the even parts removed from
\(\widetilde{\pi}\) the \emph{loose parts}, and denote by \(\pi_e\)
the partition consisting of the loose parts. Note that
\( \pi_{\mathrm{min}} \) has exactly \( k \) odd parts, 
\( |\pi_{\mathrm{min}}| + |\pi_e| = n \), and \( \pi_{\mathrm{min}} \in \widetilde{\mathcal{P}}_a(n - |\pi_e|,k) \). This step is illustrated in
Figure~\ref{fig:ground_to_E_2}.

\begin{figure}[H]
    \centering
    \begin{minipage}{\textwidth}
    \[
    \begin{aligned}
      &\underbrace{(1,1,1,\overline{2},2,2,\overline{4},4,\overline{6},7,7,7,8,8,9,9,10,11)}_{\widetilde{\pi}}
        \quad\mapsto\quad
      \Bigl(
        \underbrace{(1,1,1,\overline{2},\overline{4},\overline{6},7,7,7,8,9,9,10,11)}_{\pi_{\mathrm{min}}}
        ,\,
        \underbrace{(2,2,4,8)}_{\pi_e}
      \Bigr)
    \end{aligned}
    \]
    \end{minipage}
    \caption{From \(\mathcal{P}_a(n,k)\) to \(\mathcal{E}(n,k)\): step 2}\label{fig:ground_to_E_2}
\end{figure}
\underline{\it{Step 3:}} This step is analogous to Step~2 of the
forward direction of the proof in \Cref{thm:main_bij1}. For each
non-overlined even part \(2i\) in \(\pi_{\mathrm{min}}\), starting from the
smallest, we add \(2\) to the \(i\) odd or overlined even parts that
precede it. Again, by definition of
\(\widetilde{\mathcal{P}}_a(n - |\pi_e|,k)\), there are always at least \(i\)
such odd or overlined even parts. We repeat this process until no
non-overlined even parts remain, and call the resulting partition
\(\pi^{(1)}\). This step is illustrated in
Figure~\ref{fig:ground_to_E_5}.
\begin{figure}[H]
    \centering
  \scalebox{0.9}{%
    \begin{tikzpicture}[baseline=(current bounding box.center), every node/.style={anchor=base}]
        \node (n1)  at (-0.15,0)   {$\Big(\,1$};
        \node (n2)  at (0.6,0) {$1$};
        \node (n3)  at (1.2,0) {$1$};
        \node (n4)  at (1.8,0) {$\overline{2}$};
        \node (n5)  at (2.4,0) {$\overline{4}$};
        \node (n6)  at (3.0,0) {$\overline{6}$};
        \node (n7)  at (3.6,0) {$7$};
        \node (n8)  at (4.2,0) {$7$};
        \node (n9)  at (4.8,0) {$7$};
        \node (n10) at (5.4,0) {$8$};
        \node (n11) at (6.0,0) {$9$};
        \node (n12) at (6.6,0) {$9$};
        \node (n13) at (7.2,0) {$10$};
        \node (n14) at (7.95,0) {$11,\,\,$};
        
        \draw[->, red, thick, bend right=40] (n10.north) to (n6.north);
        \draw[->, red, thick, bend right=40] (n10.north) to (n7.north);
        \draw[->, red, thick, bend right=45] (n10.north) to (n8.north);
        \draw[->, red, thick, bend right=50] (n10.north) to (n9.north);
        \draw[->, red, thick, bend right=55] (n13.north) to (n12.north);
        \draw[->, red, thick, bend right=55] (n13.north) to (n11.north);
        \draw[->, red, thick, bend right=55] (n13.north) to (n9.north);
        \draw[->, red, thick, bend right=55] (n13.north) to (n8.north);
        \draw[->, red, thick, bend right=55] (n13.north) to (n7.north);
        \node at (3.95,-0.1) {\(\underbrace{\hspace{8.2cm}}_{\pi_{\mathrm{min}}}\)};

        \node at (9.5,0) {$\underbrace{(2,2,4,8)}_{\pi_e} \,\,\Bigg)$};

        \node at (10.75,0) {$\displaystyle \mapsto$};
        \node (r1)  at (15.5,0)  {$      \Bigg(
      \underbrace{(1,1,1,\overline{2},\overline{4},\overline{8},11,11,11,11,11,11)}_{\pi^{(1)}},\quad
      \underbrace{(2,2,4,8)}_{\pi_e}
      \Bigg)$};
      \end{tikzpicture}
    }
\caption{From \(\mathcal{P}_a(n,k)\) to \(\mathcal{E}(n,k)\): step 3}
\label{fig:ground_to_E_5}
\end{figure}

\underline{\it{Step 4:}} We construct the overpartition \(\pi^{(2)}\)
from \(\pi^{(1)}\) and \(\pi_e\), as before. For each part \(2i\) in
\(\pi_e\), starting from the largest, we add \(2\) to the largest
\(i\) odd or overlined even parts in \(\pi^{(2)}\). Repeating this
process yields an overpartition \(\pi^{(2)}\) of the same size as
\(\pi\). This step is illustrated in Figure~\ref{fig:ground_to_E_4}.

\begin{figure}[H]
\centering
  \begin{align*}
      \left(
        \underbrace{(1,1,1,\overline{2},\overline{4},\overline{8},11,11,11,11,11,11)}_{\pi^{(1)}}
        \quad,\quad
        \underbrace{(2,2,4,8)}_{\pi_e}
      \right)
      \quad \mapsto \quad
      \underbrace{(1,1,1,\overline{2},\overline{4},\overline{8},11,11,13,13,15,19)}_{\pi^{(2)}}
  \end{align*}
\caption{From \(\mathcal{P}_a(n,k)\) to \(\mathcal{E}(n,k)\): step 4}
\label{fig:ground_to_E_4}
\end{figure}

Finally, we obtain \( \lambda \) in \(\mathcal{E}(n,k)\) from
\(\pi^{(2)}\) by removing all overlines from its (distinct) even parts.

The construction of the inverse map is similar to the inverse map in the proof of \Cref{thm:main_bij1}.

\underline{\it{Inverse map, Step 1:}} Let
\(\lambda=(\lambda_1\leq\cdots\leq\lambda_\ell)\) be a partition in
\(\mathcal{E}(n,k)\). We construct an overpartition \(\lambda_s\) from
\(\lambda\) as follows. If \(\lambda_1\) is odd, then
\((\lambda_s)_1=1\); if \(\lambda_1\) is even, then
\((\lambda_s)_1=2\). For \( 1 \leq i \leq \ell-1 \), if \(\lambda_i\)
and \(\lambda_{i+1}\) are both odd, then
\((\lambda_s)_{i+1}=(\lambda_s)_i\); if \(\lambda_i\) and
\(\lambda_{i+1}\) have different parity, then
\((\lambda_s)_{i+1}=(\lambda_s)_i+1\); if \(\lambda_i\) and
\(\lambda_{i+1}\) are both even, then
\((\lambda_s)_{i+1}=(\lambda_s)_i+2\). After determining all sizes of
parts of \( \lambda_s \), we place an overline on every even part to
complete the construction. This step is illustrated in
Figure~\ref{fig:E_to_grounded_1}.
\begin{figure}[H]
\centering
\begin{minipage}{15cm}
\[
(1,1,1,2,4,8,11,11,13,13,15,19) \quad \mapsto \quad (1,1,1,\overline{2},\overline{4},\overline{6},7,7,7,7,7,7)
\]
\end{minipage}
\caption{From \(\mathcal{E}(n,k)\) to \(\mathcal{P}_a(n,k)\): step 1}
\label{fig:E_to_grounded_1}
\end{figure}

\underline{\it{Step 2:}} We construct an array of \( 2 \)s beneath the
overpartition \( \lambda_s \), arranged so that the sum of each column
gives the corresponding part of the original overpartition
\( \lambda \). This array encodes how much each odd or overlined even
part in \( \lambda_s \) must be increased by to reconstruct
\( \lambda \). See Figure~\ref{fig:E_to_grounded_2}.
\begin{figure}[h]
\centering
\begin{minipage}{15cm}
\[
(1,1,1,2,4,8,11,11,13,13,15,19) \quad \mapsto \quad \begin{array}{ccccccccccccc}
1 & 
1 & 
1 & 
\overline{2} & 
\overline{4} & 
\overline{6} & 
7 & 
7 & 
7 & 
7 & 
7 & 
7 \\
 & & & & & \textcolor{red}{2} & 
 \textcolor{red}{2} & 
 \textcolor{red}{2} & 
 \textcolor{red}{2} &
 \textcolor{red}{2} &
 \textcolor{red}{2} & 
 \textcolor{red}{2} \\
  & & & & & & \textcolor{red}{2} & 
 \textcolor{red}{2} &
 \textcolor{red}{2} & 
 \textcolor{red}{2} &
 \textcolor{red}{2} &
 \textcolor{red}{2} \\
 & & & & & & & & \textcolor{red}{2} & \textcolor{red}{2} & \textcolor{red}{2} & \textcolor{red}{2} \\
 & & & & & & & & & & \textcolor{red}{2} & \textcolor{red}{2} \\
 & & & & & & & & & & & \textcolor{red}{2} \\
 & & & & & & & & & & & \textcolor{red}{2}
\end{array}
\]
\end{minipage}
\caption{From \(\mathcal{E}(n,k)\) to \(\mathcal{P}_a(n,k)\): step 2}
\label{fig:E_to_grounded_2}
\end{figure}

\underline{\it{Step 3:}} This is a recursive construction. At each
iteration, let \( m \) be the number of columns in the array of $2$s,
and let \( 2j-1 \) be the largest part in the current partition. In
Figure~\ref{fig:E_to_grounded_2}, we have \(m=7\) and \(j=4\). We omit
the detailed description of this recursive construction, as it is
completely analogous to the map in Step~3 of the inverse map in the
proof of Theorem~\ref{thm:main_bij1}. The inserted non-loose even
parts are non-overlined. The two iterations in our running examples are
illustrated in Figures~\ref{fig:E_to_grounded_3.1}
and~\ref{fig:E_to_grounded_3.2}. From this procedure, we obtain the
pair \( (\pi_{\mathrm{min}}, \pi_e) \), where \( \pi_{\mathrm{min}} \)
is the overpartition in the top row, and \( \pi_e \) is the partition
whose parts are given by the row sums of the remaining array of
\( 2 \)s.

\begin{figure}[H]
\centering
\begin{minipage}{15cm}
\[
\mapsto \quad \begin{array}{ccccccccccccc}
1 & 
1 & 
1 & 
\overline{2} & 
\overline{4} & 
\overline{6} & 
7 & 
7 & 
7 &
8 & 
9 &
9 &
9 \\
  & & & & & & \textcolor{red}{2} & 
 \textcolor{red}{2} &
 \textcolor{red}{2} & & 
 \textcolor{red}{2} &
 \textcolor{red}{2} & \textcolor{red}{2} \\
 & & & & & & & & \textcolor{red}{2} & & \textcolor{red}{2} &  \textcolor{red}{2} & \textcolor{red}{2} \\
 & & & & & & & & & & &  \textcolor{red}{2} & \textcolor{red}{2} \\
 & & & & & & & & & & & & \textcolor{red}{2} \\
 & & & & & & & & & & & & \textcolor{red}{2}
\end{array}
\]
\end{minipage}
\caption{From \(\mathcal{E}(n,k)\) to \(\mathcal{P}_a(n,k)\): step 3, first iteration}
\label{fig:E_to_grounded_3.1}
\end{figure}

\begin{figure}[H]
\centering
\begin{minipage}{15cm}
\[
\mapsto \quad \begin{array}{cccccccccccccc}
1 & 
1 & 
1 & 
\overline{2} & 
\overline{4} & 
\overline{6} & 
7 & 
7 & 
7 &
8 &
9 &
9 &
10 &
11 \\
 & & & & & & & & \textcolor{red}{2} & & \textcolor{red}{2} &  \textcolor{red}{2} & & \textcolor{red}{2} \\
 & & & & & & & & & & &  \textcolor{red}{2} & & \textcolor{red}{2} \\
 & & & & & & & & & & & & & \textcolor{red}{2} \\
 & & & & & & & & & & & & & \textcolor{red}{2}
\end{array}
\]
\end{minipage}
\caption{From \(\mathcal{E}(n,k)\) to \(\mathcal{P}_a(n,k)\): step 3, second iteration}
\label{fig:E_to_grounded_3.2}
\end{figure}

\underline{\it{Step 4:}} By construction, the overpartition
\( \lambda_s \) obtained in Step~1 of the inverse map lies in
$\widetilde{\mathcal{P}}_a(n,k)$. After applying Steps~2 and 3, the
resulting overpartition \( \pi_{\mathrm{min}} \) still remains in
\(\widetilde{\mathcal{P}}_a(n,k)\). We then construct an overpartition
\( \widetilde{\pi} \in \widetilde{\mathcal{P}}_a(n,k) \) from
\( \pi_{\mathrm{min}} \) by inserting all parts of \( \pi_e \) into
appropriate positions, so that the resulting overpartition remains an
element of \( \widetilde{\mathcal{P}}_a(n,k) \). This is exactly the
reverse of the procedure described in Figure~\ref{fig:ground_to_E_2}.
Hence, we obtain the grounded partition
\( \pi \in \mathcal{P}_a(n,k)\) from \( \widetilde{\pi} \) by the
bijection described in Step 1 of the forward direction.

In the forward map, the procedure for obtaining
\( \lambda \in \mathcal{E}(n,k) \) from
\( \pi \in \mathcal{P}_a(n,k) \) can be written as
\[
  \pi \longmapsto \widetilde{\pi} \longmapsto (\pi_{\mathrm{min}}, \pi_e) \longmapsto \lambda.
\]
The first two maps are clearly reversible. By the construction
of the inverse map, one can easily see that the third map is also
reversible. Since every map in the process is reversible, the two sets
\( \mathcal{P}_a(n,k) \) and \( \mathcal{E}(n,k) \) are in
bijection, which completes the proof.
\end{proof}

By \Cref{thm:main_bij2} and the generating function
\eqref{eq:gf_dist_even} with \( t=1 \), we immediately obtain a
bijective proof of \eqref{eq:P2a}.

We now prove \eqref{eq:P2a_id} using properties of grounded partitions
in \( \mathcal{P}_{2,a} \).

\begin{lemma}\label{lem:gf_P2a_sum}
  We have
  \[
    \sum_{\pi \in \mathcal{P}_{2,a}} t^{\odd(\pi)} q^{|\pi|}
    =  \sum_{n=0}^\infty\frac{q^{n(n+1)}\,(-t q;q^2)_n}{(q^2;q^2)_{n}\,(t q;q^2)_{n+1}}.
  \]
\end{lemma}
\begin{proof}
This is similar to the proof of Lemma~\ref{lem:gf_P2b_sum}.

 Let \( 2n \) be the largest even part in
\( \pi \in \mathcal{P}_{2,a} \). Then each even part \( 2,4,\dots, 2n \)
  appears at least once, contributing the factor
  \( q^{n(n+1)}/(q^2;q^2)_n \) without keeping track of the colours.

  In a grounded partition $\pi \in \mathcal{P}_{2,a} $, all odd parts are coloured $b$, and all even parts are coloured $a$ or $c$, alternatingly along each number $2i$ for \( 1\leq i \leq n \). Thus the colour of the first occurrence of $2i$ determines the colour of all parts $2i$. One can encode the colour of the first $2i$ by overlining or leaving non-overlined the preceding odd number $2i-1$ (whether or not it appears as a part). This contributes the factor
  \( (-tq;q^2)_n\). The odd parts in $\pi$ are generated by \( 1/(tq;q^2)_{n+1} \), without keeping track of the colours. Combining these contributions yields the desired generating function.
\end{proof}

\begin{corollary}\label{thm:P2a_id2}
We have
\begin{align}\label{eq:identity22refined}
\sum_{n\geq0}\frac{q^{n(n+1)}\,(-t q;q^2)_n}{(q^2;q^2)_{n}\,(t q;q^2)_{n+1}}\;=\;\frac{(-q^2;q^2)_\infty}{(t q;q^2)_\infty}.
\end{align}
\end{corollary}
\begin{proof}
  It follows from \Cref{lem:gf_P2a_sum} and \eqref{eq:gf_dist_even},
  and Theorem~\ref{thm:main_bij2}.
\end{proof}

\begin{remark}\label{rem:1}
  The \( q \)-series identities in~\Cref{thm:P2b_id2,thm:P2a_id2} can
  also be proved using basic hypergeometric series. In fact, they are
  special cases of \cite[Exercise~1.6(ii)]{GR2004}. Moreover, both
  identities can also be obtained as special cases of the
  \( q \)-Gauss summation, which was interpreted combinatorially
  in~\cite{Yee2004}.
\end{remark}

\let\clearpage\oldclearpage

\printbibliography

\end{document}